\documentclass[12pt]{amsart}
\usepackage{ amsmath, amsthm, amsfonts, amssymb, color, xcolor}
 \usepackage{mathrsfs}
\usepackage{amsfonts, amsmath}
 \usepackage{amsmath,amstext,amsthm,amssymb,amsxtra,verbatim}
 \allowdisplaybreaks
\begin{document}
\baselineskip 18pt
\hfuzz=6pt

\newtheorem{theorem}{Theorem}[section]
\newtheorem{prop}[theorem]{Proposition}
\newtheorem{lemma}[theorem]{Lemma}
\newtheorem{definition}[theorem]{Definition}
\newtheorem{cor}[theorem]{Corollary}
\newtheorem{example}[theorem]{Example}
\newtheorem{remark}[theorem]{Remark}
\newcommand{\ra}{\rightarrow}
\renewcommand{\theequation}
{\thesection.\arabic{equation}}
\newcommand{\ccc}{{\mathcal C}}
\newcommand{\one}{1\hspace{-4.5pt}1}

\newcommand{\wh}{\widehat}
\newcommand{\f}{\frac}
\newcommand{\df}{\dfrac}
\newcommand{\sgn}{\textup{sgn\,}}
 \newcommand{\rn}{\mathbb R^n}
  \newcommand{\si}{\sigma}
  \newcommand{\ga}{\gamma}
   \newcommand{\nf}{\infty}
\newcommand{\p}{\partial}
\newcommand{\De}{\Delta}

\newcommand{\norm}[1]{\left\|{#1}\right\|}%
\newcommand{\supp}{\operatorname{supp}}

\newcommand{\tf}{\tfrac}
\newcommand{\lab}{\label}
\newcommand{\zzz}{\mathbf Z}
\newcommand{\li}{L^{\infty}}
\newcommand{\intrn}{\int_{\rn}}
\newcommand{\q}{\quad}
\newcommand{\qq}{\quad\quad}
\newcommand{\qqq}{\quad\quad\quad}


\newcommand{\vp}{\varphi}
\newcommand{\al}{\alpha}
\newcommand{\R}{\RR}
\newcommand{\intr}{\int_{\R}}
\newcommand{\intrr}{\int_{\R^2}}
\newcommand{\de}{\delta}
\newcommand{\om}{\omega}
\newcommand{\Tht}{\theta}
\newcommand{\tht}{\theta}
\newcounter{question}
\newcommand{\qt}{%
        \stepcounter{question}%
        \thequestion}
\newcommand{\bq}{\fbox{Q\qt}\ }
\renewcommand{\wp}{\psi}
\newcommand{\sgg}{\si_{gg}}
\newcommand{\sbtj}{\si_{b2j}}
\newcommand{\sbrk}{\si_{b3k}}
\newcommand{\sbzg}{\si_{b_0g}}

\newcommand{\abs}[1]{\left\vert #1\right\vert}%

\newcommand{\be}{\beta}%
\newcommand{\ve}{\varepsilon}%
\newcommand{\beq}{\begin{equation}}%
\newcommand{\eeq}{\end{equation}}

\def\RR{\mathbb R}
\def\bbr{\mathbb R}

\title[The Marcinkiewicz Multiplier Theorem]
{The Marcinkiewicz multiplier theorem revisited}

\thanks{  }

\author{Loukas Grafakos}

\address{Department of Mathematics, University of Missouri, Columbia MO 65211, USA}
\email{grafakosl@missouri.edu}

\author{Lenka Slav\'ikov\'a}
\address{Department of Mathematics, University of Missouri, Columbia MO 65211, USA}
\email{slavikoval@missouri.edu}

\thanks{The first   author   acknowledges the  support of the Simons Foundation and of the University  of Missouri Research Board.}

\begin{abstract}
We provide a complete proof of an optimal version of the Marcinkiewicz multiplier theorem.
\end{abstract}

\maketitle


\bigskip


\maketitle

\section{Introduction and Statement of Results} 
We revisit a product-type Sobolev space version of the  Marcinkiewicz multiplier theorem.   A version of this result first appeared in   Carbery~\cite{C} but a stronger version of it 
is a consequence of the work of  Carbery and Seeger~\cite{CSTAMS}. 
In this note we provide a self-contained proof of the  Marcinkiewicz multiplier theorem, we  point out that the  conditions on the indices   are optimal, and we provide a comparison with the H\"ormander multiplier theorem, which indicates that the former is  stronger than the latter. 

Given a bounded function $\si$ on $\rn$, we define a   linear operator
$$
T_\si(f)(x) = \int_{\rn} \wh{f}(\xi) \si(\xi) e^{2\pi i x\cdot \xi}d\xi
$$
acting on Schwartz functions $f$  on $\rn$; here 
$\wh{f}(\xi) = \int_{\rn} f(x)   e^{-2\pi i x\cdot \xi}dx$ is the Fourier transform of $f$.  
An old problem  in harmonic analysis is to find optimal  sufficient  conditions on $\si$ so that the 
operator $T_\si$ admits a bounded extension from $L^p(\rn)$ to itself for a given $p\in (1,\nf)$. 
If this is the case for a given $\si$, then we say that $\si$ is an $L^p$ Fourier multiplier.

We recall the classical Marcinkiewicz multiplier theorem:  Let 
$I_j= (-2^{j+1}, -2^j]\cup [2^j, 2^{j+1})$ for $j\in \mathbb Z$. 
Let $\si$ be a bounded  function
on $\rn$ such that for all $\al=(\al_1,\dots , \al_n)$ with $|\al_1|, \dots ,|\al_n|\le 1$ the derivatives $\p^\al \si$ are continuous up to
the boundary of any rectangle $I_{j_1}\times \cdots \times I_{j_n}$ on $\rn$.  Assume
 that  there is a   constant $A<\nf $ such that for all partitions
 $ \{s_1,\dots , s_k\}\cup\{r_1,\dots , r_\ell\}=\{1,2,\dots , n\}$ with $n=k+\ell$
 and all $\xi_{j_i}\in I_{j_i} $ we have
\begin{equation}\lab{5.2.mar-cond}
 \sup_{\xi_{r_1}\in I_{j_{r_1}}}\cdots
  \sup_{\xi_{r_\ell}\in I_{j_{r_\ell}}}
\int_{I_{j_{s_1}}} \!\! \cdots\! \int_{I_{j_{s_k}} }
\!\! \!\big |  (\p_{ {s_1}}  \cdots \p_{ {s_k}} \si ) (\xi_1, \dots , \xi_n)\big|  \,
d\xi_{ {s_k}}\cdots d\xi_{ {s_1}} \le   A
\end{equation}
for all $ (j_1,\dots , j_n)\in \mathbb Z^n$.
Then $\si$ is an $L^p$ Fourier multiplier on $\rn$ whenever $1<p<\nf$,
and there is a constant $C_{p, n}<\nf$ such that
\begin{equation}\lab{5.2.mar-cond-567}
 \|T_\si \|_{L^p\to L^p}\le C_{p, n}\big(A+ \|\si \|_{\li}\big) . 
\end{equation}

To obtain a Sobolev space version of this result  we define
$(I-\p_{{ \ell}}^2)^{\f{\ga_{\ell}}{2}}f  $
as the linear operator $ ((1+4\pi^2|\eta_{\ell}|^2)^{\f{\ga_{\ell}}{2}}\wh f(\eta))^{\vee} $ 
associated with   the multiplier
$(1+4\pi^2|\eta_{\ell}|^2)^{\f{\ga_{\ell}}{2}}$.  
The purpose of  this note is to  provide a self-contained exposition  
of a   version of the Marcinkiewicz multiplier theorem which   requires only $1/r+\varepsilon$ derivatives per variable in $L^r$ 
  instead of a full derivative in $L^1$ as in \eqref{5.2.mar-cond}.  

 \begin{theorem}\label{ThmMain}
Let $n\in \mathbb N$, $n\geq 2$. 
Suppose that $1\le r \le \infty$  and $\psi$ is a Schwartz function on the line 
whose Fourier transform is supported in $[-2,-1/2]\cup [1/2,2]$ and which satisfies 
$\sum_{j\in \mathbb Z} \wh{\psi}(2^{-j}\xi)=1$ for all $\xi\neq 0$. Let $\ga_\ell>1/r$, $\ell=1,\dots , n$. If  
a function   $\si$ on $\rn$ satisfies
\begin{equation}\label{20000}
\sup_{ k_1,\dots  , k_n\in \mathbb Z}   
  \big\|(I-\p_{{ 1}}^2)^{\f {\ga_1}2} \cdots(I-\p_{{ n}}^2)^{\f {\ga_n}2} \big( \wh{\psi}(\xi_1)\cdots 
\wh{\psi}(\xi_n) \si (2^{k_1} \xi_1, \dots , 2^{k_n}\xi_n)\big) \big\|_{L^r }<\infty,
\end{equation}
then $T_\si$ admits a bounded extension from
$L^p(\rn)$ to itself for all $1<p<\nf$ with
\begin{equation}\label{Marc-con}
\Big| \f1p -\f12\Big| < \min (\ga_1, \dots , \ga_n). 
\end{equation}  
Moreover,  \eqref{Marc-con} is optimal in the sense that if  $T_\si$ is $L^p$-bounded for every $\sigma$ satisfying~\eqref{20000}, then  the strict inequality in \eqref{Marc-con} must necessarily hold.
\end{theorem}

Carbery~\cite{C} first formulated a version of Theorem~\ref{ThmMain} in which the 
 multiplier lies in a product-type $L^2$-based Sobolev space.   
 Carbery and Seeger~\cite[Remark after Prop. 6.1]{CSTAMS} obtained 
 Theorem~\ref{ThmMain} in the case when 
 $\gamma_1=\cdots = \gamma_n >\big| \f1p -\f12\big|=\f1r$.
  The positive direction of their result also  appeared in~\cite[Condition (1.4)]{CS} but this time   the range of $p$ is
 $\big| \f1p -\f12\big| < \f1r$ and is  
  expressed in  terms of the integrability of the  multiplier and not in terms of its smoothness. In our 
  opinion  it is more natural,  nonetheless, to 
 state condition \eqref{Marc-con} in terms of the smoothness of the multiplier, as in the case of the 
 sharp version of the H\"ormander multiplier theorem, see  \cite{GrafSlav1}.  

The class of multipliers which satisfies the assumptions of Theorem~\ref{ThmMain} is strictly larger than the set of multipliers treated by the version of the H\"ormander multiplier theorem due to Calder\'on and Torchinsky~\cite[Theorem 4.6]{CT}; see also
\cite{GraHeHonNg1}. Since   Theorem~\ref{ThmMain} assumes the multiplier $\sigma$ to have $1/r+\varepsilon$ derivatives in each variable, while the H\"ormander multiplier theorem requires more than $n/r$ derivatives in all variables, it is apparent that there are multipliers which can be treated by Theorem~\ref{ThmMain}, but not by~\cite[Theorem 4.6]{CT}. On the other hand, it is an easy consequence of the following theorem that every multiplier satisfying the assumptions of the H\"ormander multiplier theorem falls under the scope of Theorem~\ref{ThmMain} as well.

\begin{theorem}\label{T:comparison}
Let $\psi$ be a Schwartz function on the line whose Fourier transform is supported in the set $\{\xi: \frac{1}{2}\leq|\xi|\leq 2\}$ and which satisfies $\sum_{k\in \mathbb Z} \wh{\psi}(2^k \xi) =1$ for every $\xi \neq 0$. Also, let $\Phi$ be a Schwartz function on $\R^n$ having analogous properties. If $1<r<\infty$ and $\gamma_1,\dots,\gamma_n$ are real numbers larger than $\frac{1}{r}$, then 
\begin{align}\label{E:inequality}
&\sup_{j_1,\dots,j_n\in \mathbb Z} \left\|(I-\partial_1^2)^{\frac{\gamma_1}{2}} \cdots (I-\partial_n^2)^{\frac{\gamma_n}{2}}\bigg[\sigma(2^{j_1}\xi_1,\dots,2^{j_n}\xi_n) \prod_{\ell=1}^n \wh{\psi}(\xi_\ell)\bigg]\right\|_{L^r}\\
&\leq C\sup_{j\in \mathbb Z} \left\|(I-\Delta)^{\frac{\gamma_{1}+\dots+\gamma_n}{2}} \bigg[\sigma(2^{j}\xi) \wh{\Phi}(\xi)\bigg]\right\|_{L^r}. 
\nonumber
\end{align}
\end{theorem}

\section{The proof of Theorem~\ref{ThmMain}; The main estimate}

Let us start this section by introducing some notation that will be used throughout the paper. We will use $\psi$ to denote the bump from Theorem~\ref{ThmMain}; further, $\theta$ will stand for the function on the line satisfying
$$
\wh{\theta}(\eta) =  \wh{\psi}(\eta/2)+ \wh{\psi}(\eta )+\wh{\psi}(2\eta ).
$$
One can observe that $\wh{\theta}$ is supported in $\{\frac{1}{4} \leq |\xi| \leq 4\}$ and $\wh{\theta}=1$ on the support of $\wh{\psi}$.

To simplify the notation, if $\xi=(\xi_1,\dots,\xi_n)\in \mathbb R^n$ and $J=(j_1,\dots,j_n)\in \mathbb Z^n$, we shall write
$$
2^J\xi=\big(2^{j_1} \xi_{1} , \dots , 2^{j_n} \xi_{n} \big)
$$
and
$$
\wh{\psi}(\xi) = \prod_{\ell=1}^n  \widehat{\psi}(\xi_{\ell}), \quad \wh{\theta}(\xi) = \prod_{\ell=1}^n  \widehat{\theta}(\xi_{\ell}).
$$

Let $k\in 1,\dots,n$. For $j\in \mathbb Z$ we define the Littlewood-Paley operators associated to the bumps $\psi$ and $\theta$ by
$$ 
\De^{k}_j f = \De^{\psi,k}_j f = \int_{\mathbb R} f(x_1,\dots,x_{k-1}, x_k-y,x_{k+1}, \dots, x_n) 2^{j} \psi (2^j y) dy
$$
and
$$ 
\De^{\theta,k}_j f = \int_{\mathbb R} f(x_1,\dots,x_{k-1}, x_k-y,x_{k+1}, \dots, x_n) 2^{j} \theta (2^j y) dy.
$$

We begin with the following lemma:

\begin{lemma}\label{L:lemma}
Let $1\le r\le \infty $, let $1\le \rho<2$ satisfy $1\leq \rho\leq r$ and let $\gamma_1, \dots ,\gamma_n$ be real numbers such that $\gamma_\ell\rho>n_\ell$, $\ell=1,\dots,n$. Then, for any function $f$ on $\R^n$ and for all integers $j_1,\dots j_n$, we have
\begin{equation}\label{E:estimate}
|\De_{j_1}^{1} \cdots \De_{j_n}^{n} T_\si(f)| \le C\, K \, 
\Big[M^{(1)}\cdots M^{(n)} (|\De_{j_1}^{\theta,1} \cdots \De_{j_n}^{\theta,n} f|^\rho) \Big]^{\f 1\rho} , 
\end{equation}
where $M^{(\ell)}$ denotes the one-dimensional Hardy-Littlewood maximal operator in the $\ell$-th coordinate and 
$$
K = \sup_{j_1,\dots , j_n\in \mathbb Z} \bigg\| (I-\partial_{1}^2)^{\frac{\gamma_1}2} \cdots  (I-\partial_{n}^2)^{\frac{\gamma_n}2}  
\Big[ \sigma \big(2^{j_1} \xi_{1} , \dots , 2^{j_n} \xi_{n} \big) \widehat{\psi}(\xi_{1}) \cdots  
\widehat{\psi}(\xi_{n})  \Big]  \bigg\|_{L^r}.
$$
\end{lemma}

\begin{proof}
Throughout the proof we shall use the notation introduced at the beginning of this section and, whenever $J=(j_1,\dots,j_n)$, we shall write 
$$
\Delta_J f=\Delta_{j_1}^{1}\cdots \Delta_{j_n}^{n} f, \quad \Delta_J^{\theta} f=\Delta_{j_1}^{\theta,1}\cdots \Delta_{j_n}^{\theta,n} f.
$$

Since $\wh{\theta}$ is equal to $1$ on the support of $\wh{\psi}$, we have 
\begin{align*}
&\De_{J} T_\si(f)(x_1,\dots, x_n) 
= \int_{\R^n} \wh{f}(\xi) \wh{\psi}(2^{-J} \xi) \si(\xi) e^{2\pi i x\cdot \xi}d\xi\\
&= \int_{\R^n} \wh{f}(\xi) \wh{\theta}(2^{-J} \xi) \wh{\psi}(2^{-J}\xi) \si(\xi) e^{2\pi i x\cdot \xi}d\xi
=\int_{\R^n} (\De_{J}^{\theta} f)\sphat{}\,(\xi) \wh{\psi}(2^{-J} \xi)  \si(\xi) e^{2\pi i x\cdot \xi}d\xi\\
&=\int_{\R^n} 2^{j_1+\cdots +j_n} (\De_{J}^{\theta} f)\sphat{}\,(2^{J} \xi') \wh{\psi}(\xi')  \si(2^{J} \xi') e^{2\pi i(2^{J}x \cdot \xi')} d\xi'\\
&=2^{j_1+\cdots +j_n} \int_{\R^n} (\De_{J}^{\theta} f)(y) \left[\wh{\psi}(\xi')  \si(2^{J} \xi')\right]\sphat\,  (2^{J}(x-y)) dy\\
&=\int_{\R^n}   \frac{2^{j_1+\cdots +j_n} (\De_{J}^{\theta} f)(y)}{\prod_{\ell=1}^n(1+2^{j_\ell}|x_{\ell}-y_{\ell}|)^{\gamma_\ell}}  
  \cdot \prod_{\ell=1}^n (1+2^{j_\ell}|x_{\ell}-y_{\ell}|)^{\gamma_\ell}
\left[\wh{\psi}(\xi)  \si(2^{J} \xi)\right]\sphat\,  (2^{J}(x-y)) \, dy.
\end{align*}
H\"older's inequality thus yields that 
 $|\De_{J} T_\si(f)(x)|$ is bounded by 
\begin{align*} 
&\bigg(\int_{\R^n} 2^{j_1+\cdots +j_n} \frac{|(\De_{J}^{\theta} f)(y)|^\rho}{\prod_{\ell=1}^n(1+2^{j_\ell}|x_{\ell}-y_{\ell}|)^{\gamma_\ell \rho}}\, dy\bigg)^{\frac{1}{\rho}}\\
&\cdot
\bigg(\int_{\R^n} 2^{j_1+\cdots +j_n} \bigg|\prod_{\ell=1}^n(1+2^{j_\ell}|x_{\ell}-y_{\ell}|)^{\gamma_\ell}
\cdot \bigg[\wh{\psi}(\xi)  \si(2^{J}\xi)\bigg]\sphat\,  (2^{J}(x-y))\bigg|^{\rho'} \, dy\bigg)^{\frac{1}{\rho'}},
\end{align*}
where, when $\rho=1$, the second term in the product is to be interpreted as  
$$
\bigg\|\prod_{\ell=1}^n(1+2^{j_\ell}|x_{\ell}-y_{\ell}|)^{\gamma_\ell}
\cdot \bigg[\wh{\psi}(\xi)  \si(2^{J}\xi)\bigg]\sphat\,  (2^{J}(x-y))\bigg\|_{L^\infty} .
$$

Since $\gamma_\ell \rho >n_\ell$ for all $\ell=1,\dots,n$, $n$ consecutive applications of~\cite[Theorem 2.1.10]{CFA} yield the estimate  
\begin{align}
 \bigg( \int_{\R^n} 2^{j_1+\cdots +j_n} \frac{|(\De_{J}^{\theta} f)(y)|^\rho}{\prod_{\ell=1}^n(1+2^{j_\ell}|x_{\ell}-y_{\ell}|)^{\gamma_\ell \rho}} \, dy\bigg)^{\frac 1\rho}   \nonumber   
\leq C \Big[ M^{(1)} \cdots M^{(n)} \big( |\De_{J}^{\theta}  f|^\rho\big)  (x) \Big]^{\f 1\rho}. 
\end{align}

We now write
\begin{align*} 
&\bigg(\int_{\R^n} 2^{j_1+\cdots +j_n} \bigg|\prod_{\ell=1}^n(1+2^{j_\ell}|x_{\ell}-y_{\ell}|)^{\gamma_\ell} 
\bigg[\wh{\psi}(\xi)\si(2^{J} \xi)\bigg]\sphat\,  (2^{J}(x-y))\bigg|^{\rho'} \!\! dy\!\!\bigg)^{\frac{1}{\rho'}}\\
&\leq   \bigg(\int_{\R^n} \bigg|\prod_{\ell=1}^n(1+|y_{\ell}|^2)^{\frac{\gamma_\ell}{2}}
\bigg[\wh{\psi}(\xi)   \si(2^{J} \xi)\bigg]\sphat\,  (y)\bigg|^{\rho'} \, dy \bigg)^{\frac{1}{\rho'}}\\
&\leq   \bigg\| (I-\partial_{1}^2)^{\frac{\gamma_1}2} \cdots  (I-\partial_{n}^2)^{\frac{\gamma_n}2}  
\Big[ \sigma \big(2^J \xi\big) \widehat{\psi}(\xi)  \Big]  \bigg\|_{L^\rho}
\leq C K .  
\end{align*}
Notice that the second inequality is   the Hausdorff-Young inequality and the last inequality 
is a consequence of the Kato-Ponce inequality \cite{GOh} (if $\rho<r$).
A combination of the preceding estimates yields \eqref{E:estimate}.
\end{proof}

\begin{prop}\label{P:first_endpoint}
Let $1\leq r\leq \infty$ and let $\gamma_\ell>\max\{1/2, 1/r\}$, $\ell=1,\dots,n$. If a function $\sigma$ on $\R^n$ satisfies~\eqref{20000}, then $T_\sigma$ admits a bounded extension from $L^p(\R^n)$ to itself for all $1<p<\infty$. 
\end{prop}

\begin{proof}
Suppose first that $p> 2$. Since $\gamma_\ell>\max\{1/2, 1/r\}$, $\ell=1,\dots,n$, we can find $\rho \in [1,2)$ 
such that $\rho \leq r$ and $\rho\gamma_\ell>1$, $\ell=1,\dots,n$. Then
\begin{align*}
\big\| T_\si(f) \big\|_{L^p(\R^n)} 
&\leq C_p(n) \Big\| \Big( \sum_{j_1,\dots,j_n \in \mathbb Z} |\De_{j_1}^{1}\cdots \De_{j_n}^{n} T_\si (f) |^2 \Big)^{\f12} \Big\|_{L^p}\\ 
&\le C_p(n)  K \Big\| \Big( \sum_{j_1,\dots,j_n \in \mathbb Z}  \Big[M^{(1)}\cdots M^{(n)} (|\De_{j_1}^{\theta,1} \cdots \De_{j_n}^{\theta,n} f|^\rho) \Big]^{\f 2 \rho}     \Big)^{\f12} \Big\|_{L^p}\\
&\leq C_p(n)  K \Big\| \Big( \sum_{j_1,\dots,j_n \in \mathbb Z}  |\De_{j_1}^{\theta,1} \cdots \De_{j_n}^{\theta,n} f|^2 \Big)^{\f12} \Big\|_{L^p}\\
&\leq C_p(n)  K \|f\|_{L^p}.
\end{align*}
Notice that the second inequality follows from Lemma~\ref{L:lemma} and the third inequality is obtained by applying the Fefferman-Stein inequality \cite{FS} on the Lebesgue space  $L^{\frac{p}{2}}$  in each of the variables $y_{1},\dots, y_{n}$. Observe that the Fefferman-Stein inequality makes use of the assumptions $2/\rho>1$ and $p/2> 1$.

The case $1<p<2$ follows by a duality argument, while the case $p=2$ is a consequence of  
 Plancherel's theorem and of a Sobolev embedding into $L^\infty$.
\end{proof}

\section{The proof of Theorem~\ref{ThmMain}; An interpolation argument}
 
When $p=2$ no derivatives are required of $\si$ for $T_\si$ to be bounded. To mitigate the effect of the 
requirement of the derivatives of $\si$ for $T_\si$ to be bounded  on $L^p$ for $p\neq 2$,  we 
apply an interpolation argument between $p=2$ and $p=1+\de$.

We shall use the notation introduced at the beginning of the previous section, and we shall denote 
$$
\Gamma\big(\{s_\ell\}_{\ell=1}^n\big) =\Gamma(s_1,\dots , s_n)= (I-\partial_1^2)^{\frac{s_1}2} \cdots  (I-\partial_n^2)^{\frac{s_n}2} .
$$
The following result will be the key interpolation estimate: 

\begin{theorem}\label{interpL} 
 Fix $1<p_0,p_1,r_0,r_1<\infty$, $0<  s_1^0, \dots , s_n^0, s_1^1, \dots , s_n^1<\infty$. Suppose that 
 $r_0s_\ell^0>1$ and $r_1s_\ell^1>1$ for all $\ell=1,\dots , n$.
 Let $ \psi $ be as before. 
  Assume that for $k\in\{0,1\}$ we have
  \begin{equation}
    \norm{T_\sigma(f)}_{L^{p_k}}\le K_k 
     \sup_{j_1,\dots , j_n\in \mathbb Z} \bigg\|  \Gamma (s_1^k,\dots , s_n^k)
\Big[ \sigma  (2^J \xi  ) \prod_{\ell=1}^n  
\widehat{\psi}(\xi_{\ell})  \Big]  \bigg\|_{L^{r_k}} 
    \norm{f}_{L^{p_k}}
  \end{equation}
  for   all $f\in \mathscr C_0^\nf(\mathbb R^n)$. 
For $0<\theta<1$ and $\ell=1,\dots , n$ define 
  $$
  \frac 1p  =\frac{1-\theta}{p_0} + \frac{ \theta}{p_1} , \quad
       \frac 1r =\frac{1-\theta}{r_0} + \frac{ \theta}{r_1} , \quad
          s_\ell =(1-\theta)s_\ell^0 +   \theta s_\ell^1  .
  $$
Then  there is  a constant $C_* $ such that for all $f\in \mathscr C_0^\nf(\mathbb R^n)$
we have
  \begin{equation}
    \norm{T_\sigma(f)}_{L^{p} }\le C_* K_0^{1-\theta} K_1^{\theta}  
     \sup_{j_1,\dots , j_n\in \mathbb Z} \bigg\|  \Gamma ( s_1 ,\dots , s_n )
\Big[ \sigma  (2^J \xi  ) \prod_{\ell=1}^n  
\widehat{\psi}(\xi_{\ell})  \Big]  \bigg\|_{L^{r }} \|f\|_{L^p}.
  \end{equation} 
\end{theorem}

Assuming Theorem~\ref{interpL}, we   complete the proof of Theorem~\ref{ThmMain} as follows:  

\begin{proof}
Given $1\leq r\leq \infty$ and $\ga_\ell >1/r$, $\ell=1,\dots,n$, fix $p\in (1,\infty)$ satisfying~\eqref{Marc-con}. In fact, we can assume that $p\in (1,2)$, since the case   $p\in (2,\infty)$     follows by duality  and the case $p=2$ is a consequence of Plancherel's theorem and of a Sobolev embedding into $L^\infty$. In addition, assume first that $\min_\ell \ga_\ell \leq \frac{1}{2}$. 
In view of~\eqref{Marc-con}, there is $\tau\in (0,1)$ such that
\begin{equation}\label{E:tau}
\frac{1}{p}-\frac{1}{2}<\tau \min_{\ell} \ga_\ell.
\end{equation}
Set $p_1=\frac{2}{\tau+1}$, $r_1=2r\min_{\ell} \ga_\ell$ and $\ga_\ell^1=\frac{1}{2}+\varepsilon$, $\ell=1,\dots,n$, where $\varepsilon>0$ is a real number whose exact value will be specified later. Since $p_1>1$ and $r_1 \ga_\ell^1>2\ga_\ell^1 >1$, $\ell=1,\dots,n$, Proposition~\ref{P:first_endpoint} yields that
\begin{equation}\label{endpoint1}
    \norm{T_\sigma(f)}_{L^{p_1}}\le C_1
    \sup_{j_1,\dots , j_n\in\mathbb Z}\bigg\|{  \Gamma ( \ga_1^1 ,\dots , \ga_n^1 )
\Big[ \sigma  (2^J \xi  ) \prod_{\ell=1}^n  
\widehat{\psi_\ell}(\xi_{\ell}) \Big]   }\bigg\|_{L^{r_1} }
    \norm{f}_{L^{p_1}}.
  \end{equation}

Pick $p_0=2$. Let $\theta$ be the real number satisfying
$$
\frac{1}{p}=\frac{1-\theta}{p_0}+\frac{\theta}{p_1},
$$
namely, $\theta=\frac{2}{\tau}(\frac{1}{p}-\frac{1}{2})$. Observe that, by~\eqref{E:tau}, $0<\theta<2\min_\ell \ga_\ell \leq 1$. Finally, choose real numbers $r_0$ and $\ga_\ell^0$, $\ell=1,\dots,n$, in such a way that 
\begin{equation}\label{E:rs}
\frac{1}{r}=\frac{1-\theta}{r_0}+\frac{\theta}{r_1}, \quad \ga_\ell=(1-\theta)\ga_\ell^0+\theta \ga_\ell^1.
\end{equation}
We claim that, for a suitable choice of $\varepsilon>0$, one has $r_0>2$ and $r_0 \ga_\ell^0>1$, $\ell=1,\dots,n$. Indeed, since $\min_\ell \ga_\ell \leq \frac{1}{2}$, we have $r_1\leq r$, and thus, by~\eqref{E:rs}, $r_0\geq r\geq r_1>2$. Further,
\begin{equation*}
r_0\ga_\ell^0= \frac{r_1 r(\ga_\ell - \theta \ga_\ell^1)}{r_1-\theta r}
=\frac{r \min_k \ga_k(\ga_\ell - \frac{\theta}{2}-\theta\varepsilon)}{\min_k \ga_k -\frac{\theta}{2}}
\geq \frac{r \min_k \ga_k(\min_k \ga_k - \frac{\theta}{2}-\theta\varepsilon)}{\min_k \ga_k -\frac{\theta}{2}}.
\end{equation*}
Since $r\min_k \ga_k >1$ and $\min_k \ga_k -\frac{\theta}{2}>0$, one gets $r_0 \ga_\ell^0>1$ if $\varepsilon>0$ is small enough. 
Consequently, the space $\{g:\,\,  \Gamma ( \ga_1^0 ,\dots , \ga_n^0 )g \in   L^{r_0}\}$ 
  embeds in $L^\nf$, and we thus have 
    \begin{equation}\label{400}
    \norm{T_\sigma(f)}_{L^{2}}\le C_1
    \sup_{j_1,\dots , j_n\in\mathbb Z}\bigg\|{  \Gamma ( \ga_1^0 ,\dots , \ga_n^0 )
\Big[ \sigma  (2^J \xi  ) \prod_{\ell=1}^n  
\widehat{\psi_\ell}(\xi_{S_\ell}) \Big]   }\bigg\|_{L^{r_{0}} }
    \norm{f}_{L^{2}}.
\end{equation}
The boundedness of $T_\sigma$ on $L^p(\R^n)$ for any $\sigma$ satisfying~\eqref{20000} thus follows from Theorem~\ref{interpL}. 
Finally, if $\min_\ell \ga_\ell>\frac{1}{2}$, then the required assertion follows directly from Proposition~\ref{P:first_endpoint}. 
\end{proof}
 
\begin{proof}[Proof of Theorem~\ref{interpL}]
  The proof of Theorem~\ref{interpL} follows closely that of ~\cite[Theorem 4.7]{CT} and for this reason we only provide an outline of its proof with few  details. Throughout the proof we shall use the notation introduced at the beginning of the previous section. Also, whenever $J\in \mathbb Z^n$, we denote
  $$
  \varphi_J =  \Gamma ( s_1 ,\dots , s_n ) \Big[ \sigma  (2^J \xi  ) \wh{\psi}(\xi)  \Big], 
  $$
  and for $z$ with real part in $[0,1]$ we define
\begin{equation}\label{100}
  \sigma_z(\xi) = \sum_{J\in \mathbb Z^n}      \Gamma \Big( \{-s_\ell^0(1-z)-s_\ell^1 z \}_{\ell=1}^n\Big)    \!
    \left[  |\varphi_J|^{r(\frac{1-z}{r_0}+\frac{z}{r_1})}e^{i \textup{Arg } (\varphi_J)}
  \right](2^{-J}\xi)   \widehat{\theta } (2^{-J}\xi).
\end{equation}
 For any given $\xi \in\R^n$, this sum has only finitely many terms and one can show that
  \begin{equation}\label{200}
\| \si_{\tau+it} \|_{\li} \lesssim (1+ \, |t|)^{\frac{3n}{2}}
         \Big(  \sup_{J\in\mathbb Z^n}\big\| \Gamma(s_1,\dots , s_n)[{\sigma(2^J\xi)\widehat{\psi}}(\xi)]\big\|_{L^{r} } \Big)^{\f r{r_\tau}} \,,
   \end{equation}
	where $r_\tau$ is the real number satisfying $\frac{1}{r_\tau}=\frac{1-\tau}{r_0}+\frac{\tau}{r_1}$.

  Let $T_z$ be the family of operators associated to the multipliers $\sigma_z.$  
  Fix $f, g\in  \mathscr C_0^\nf$ and $1<p_0<p<p_1<\nf$.
Given $\ve>0$ there exist functions $f_z^{ \ve}$ and
  $ {g}_z^{ \ve}$
  such that
  $  \norm{f_\theta^\ve-f}_{L^p}<\ve$, $ \norm{g_\theta^\ve-g}_{L^{p'}}<\ve, $ 
  and that
  \begin{align*}
 & \norm{f_{it}^\ve}_{L^{p_0}}\le   \big(\! \norm{f}_{L^{p}}^p+\ve\big)^{\frac 1{p_0}} ,\quad
  \norm{f_{1+it}^\ve}_{L^{p_1}}\le   \big(\!  \norm{f}_{L^{p}}^p+\ve\big)^{\frac 1{p_1}},\\
 & \norm{g_{it}^\ve}_{L^{p_0'}}\le   \big(\!  \norm{g}_{L^{p'}}^{p'}+\ve\big)^{\frac {1}{p_0'}},\quad
  \norm{g_{1+it}^\ve}_{L^{p_1'}}\le  \big(\!  \norm{g}_{L^{p'}}^{p'}+\ve\big)^{\frac {1}{p_1'}}.
  \end{align*} 
 The existence of $f_z^{ \ve}$ and
  $ {g}_z^{ \ve}$ is folklore and is omitted; for a similar construction see~\cite[Theorem 3.3]{CT}. 
Let $  F(z) =  \int T_{\sigma_z}(f_z^\ve) {g}_z^\ve\; dx $. Then $F(z)$ is equal to 
  {\allowdisplaybreaks
  \begin{align*}
  &\int_{\R^n} \sigma_z(\xi)\widehat{f_z^\ve}(\xi)\widehat{g_z^\ve}(\xi)\; d\xi\\
  =& \sum_{J\in \mathbb Z^n}\int_{\R^n}    \Gamma \Big( \{-s_\ell^0(1-z)-s_\ell^1 z \}_{\ell=1}^n\Big)    \! 
  \left[
  |\varphi_J|^{r(\frac{1-z}{r_0}+\frac{z}{r_1})}e^{i \textup{Arg } (\varphi_J)}
  \right](2^{-j}\xi)\widehat{\theta}(2^{-J}\xi)
  \widehat{f_z^\ve}(\xi)\widehat{ {g}_z^\ve}(\xi)\; d\xi\\
  =& \sum_{J\in \mathbb Z^n}\int_{\R^n}
  \left[|{\varphi_J}|^{r(\frac{1-z}{r_0}+\frac{z}{r_1})}e^{i \textup{Arg } (\varphi_J)}
  \right](2^{-J}\xi)   \Gamma \Big( \{-s_\ell^0(1-z)-s_\ell^1 z \}_{\ell=1}^n\Big)    \!
  \left[\widehat{\theta}(2^{-J}\cdot)
  \widehat{f_z^\ve}\widehat{   {g}_z^\ve}\right]\!(\xi)\; d\xi.
  \end{align*}
  }

 The function  $F(z)$ is analytic on the strip $0<\Re(z)<1$ and continuous up to the boundary. 
Notice that    $\sigma_{it} (2^K\cdot) \wh\psi$  picks up only the terms of \eqref{100} for which $J$ differs from $K$ 
in some coordinate by at most one unit.  For simplicity we may therefore take $K=J$ in the calculation below. 
  Using the Kato-Ponce inequality we may ``remove" the factor $\wh\theta$  and   write 
  \begin{align*}
&\hspace{-16pt}  \norm{T_{\sigma_{it}}(f_{it}^\ve)}_{L^{p_0}}  \\
  \le& K_0\sup_{K\in\mathbb Z^n}
  \norm{ \Gamma(s_1^0,\dots , s_n^0) \big[\sigma_{it}(2^K\cdot)\widehat{\psi}\big]}_{L^{r_0} }\norm{f_{it}^\ve}_{L^{p_0}}\\
  \le & K_0\sup_{K\in\mathbb Z^n} 
  \big\|  \Gamma \big( \{ s_\ell^0  -s_\ell^0(1-it)-s_\ell^1 it \}_{\ell=1}^n\big)    \!
  \big[ |\varphi_K|^{r(\frac{1-it}{r_0}+\frac{it}{r_1})}e^{i \textup{Arg } (\varphi_K)}
  \big] \big\|_{L^{r_0} }\norm{f_{it}^\ve}_{L^{p_0}} \\
  \lesssim &  (1+  |t|)^{\frac{3n}2}K_0\,
  \sup_{K\in\mathbb Z^N}\| {\varphi_K}\|_{L^r}^{\frac{r}{r_0}}
   \big( \norm{f}_{L^{p}}^p+\ve\big)^{\frac 1{p_0}} .
  \end{align*}
Using H\"older's inequality  $  \abs{F(it)}\le \norm{T_{\sigma_{it}}(f_{it}^\ve)}_{L^{p_0}}\norm{g_{it}^\ve}_{L^{p_0'}}, $ 
 we may therefore write 
  $$
  \abs{F(it)}\le C  (1+\abs{t})^{\frac{3n}2}K_0
  \sup_{J\in\mathbb Z^n}
   \big\|{\Gamma(\{s_\ell\}_{\ell=1}^n)[\sigma(2^J\cdot)\widehat{\psi}]} \big\|_{L^r}^{\frac{r}{r_0}}
 \big( \norm{f}_{L^{p}}^p+\ve\big)^{\frac 1{p_0}}
  \big( \norm{g}_{L^{p'}}^{p'}+\ve\big)^{\frac {1}{p_0'}} 
  $$
for some constant $C=C(n,r_0,s_\ell^0,s_\ell^1)$.  Similarly, for some constant $C=C(n,r_1,s^0_\ell,s^1_\ell)$ we   obtain
  $$
  \abs{F(1+it)}\le C  (1+\abs{t})^{\frac{3n}2}K_1
  \sup_{J\in\mathbb Z^n}
   \big\|{\Gamma(\{s_\ell\}_{\ell=1}^n)[\sigma(2^J\cdot)\widehat{\psi}]} \big\|_{L^r}^{\frac{r}{r_1}}
\big( \norm{f}_{L^{p}}^p+\ve\big)^{\frac 1{p_1}}
  \big( \norm{g}_{L^{p'}}^{p'}+\ve\big)^{\frac 1{p_1'}}.
  $$
Thus for $z=\tau+it$,  $t\in\mathbb{R}$  and $0\le \tau \le  1$
it follows from \eqref{200} and from the definition of $F(z)$  that
$$
|F(z)|\le 
C'' 
(1+ \, |t|)^{\frac{3n}2}
         \Big(   \sup_{J\in\mathbb Z^n}\big\|\Gamma (s_1,\dots , s_n)\big[{\sigma(2^J\cdot)\widehat{\psi}}\big]\big\|_{L^{r} } \Big)^{\f r{r_\tau}}
        \|f_z^\ve\|_{L^2}   \|g_z^\ve\|_{L^2} 
=A_\tau(t)\, ,
$$
noting that $  \|f_z^\ve\|_{L^2}   \|g_z^\ve\|_{L^2} $ is bounded above by constants independent of $t$ and $\tau$. 
Since $A_\tau(t)\le
\exp(A e^{a|t|})  $, the hypotheses of three lines lemma are valid.
It follows that 
  $$
  \abs{F(\theta)}\le C\, K_0^{1-\theta} K_1^{\theta}\sup_{J\in\mathbb Z^n}
  \big\|{\Gamma(\{s_\ell\}_{\ell=1}^n)[\sigma(2^J\cdot)\widehat{\psi}]} \big\|_{L^r}
 \big( \norm{f}_{L^{p}}^p+\ve\big)^{\f{1}{p}}
  \big( \norm{g}_{L^{p'}}^{p'}+\ve\big)^{\f{1}{p'}} .
  $$
Taking the supremum over all functions $g\in L^{p'}$ with $\| g\|_{L^{p'}} \le 1$, a simple density argument yields
for some $C_*=C_*(n,r_1,r_2,s_\ell^0,s_\ell^1)$
  $$
  \norm{T_{\sigma }(f)}_{L^p}\le C_*\, K_0^{1-\theta} K_1^\theta \sup_{J\in\mathbb Z^n}
  \big\|{ \Gamma(s_1,\dots , s_n) [\sigma(2^J\cdot)\widehat{\psi}]}\big\|_{L^r}
\norm{f}_{L^{p}}. 
  $$
This completes the proof of the sufficiency part of Theorem~\ref{ThmMain}. The proof of the necessity part is postponed to section~\ref{S:examples}. 
	\end{proof}

\section{The proof of Theorem~\ref{T:comparison}}

Crucial ingredients needed for the proof of Theorem~\ref{T:comparison} are two one-dimensional inequalities contained in the following lemma.

\begin{lemma}\label{L:identity}
Let $\psi$ be as in Theorem~\ref{T:comparison}. If $k\in \mathbb Z$, $\gamma>0$ and $1\leq r \leq \infty$ are such that $\gamma r>1$, then
\begin{equation}\label{E:identity}
\big\|f(2^k \cdot) \wh{\psi}\big\|_{L^r(\R)} \leq C \big\|(I-\partial^2)^{\frac{\gamma}{2}} f\big\|_{L^r(\R)}
\end{equation}
and
\begin{equation}\label{E:laplacian}
\big\|(-\partial^2)^{\frac{\gamma}{2}}\big[f(2^{k}\cdot)\wh{\psi}\big]\big\|_{L^r(\R)}
\leq C\big(1+2^{k(\gamma-\frac{1}{r})}\big)\big\|(I-\partial^2)^{\frac{\gamma}{2}}f\big\|_{L^r(\R)}.
\end{equation}
\end{lemma}

\begin{proof}
Since $\gamma r>1$, the Sobolev embedding theorem yields
\begin{equation}\label{E:sobolev_embedding}
|f(2^k x)|\leq \|f\|_{L^\infty(\R)}\leq  C \big\|(I-\partial^2)^{\frac{\gamma}{2}} f\big\|_{L^r(\R)} \quad \textup{for a.e. } x\in \R.
\end{equation}
Therefore,
$$
\big\|f(2^k \cdot) \wh{\psi}\big\|_{L^r(\R)} 
\leq C \big\|(I-\partial^2)^{\frac{\gamma}{2}} f\big\|_{L^r(\R)} \big\|\wh{\psi}\big\|_{L^r(\R)} 
=C' \big\|(I-\partial^2)^{\frac{\gamma}{2}} f\big\|_{L^r(\R)}.
$$
This proves~\eqref{E:identity}.

Further, using the Kato-Ponce inequality, the estimate~\eqref{E:sobolev_embedding} and the fact that $\wh{\psi}$ is smooth and with compact support, we obtain
\begin{align*}
&\big\|(-\partial^2)^{\frac{\gamma}{2}}\big[f(2^{k}\cdot)\wh{\psi}\big]\big\|_{L^r(\R)}\\
&\leq C\left(\big\|(-\partial^2)^{\frac{\gamma}{2}}\big[f(2^{k}\cdot)\big]\big\|_{L^r(\R)} \big\|\wh{\psi}\big\|_{L^\infty(\R)} +\big\|f(2^{k}\cdot)\big\|_{L^\infty(\R)}\big\|(-\partial^2)^{\frac{\gamma}{2}}\wh{\psi}\big\|_{L^r(\R)}\right)\\
&\leq C\left(\big\|(-\partial^2)^{\frac{\gamma}{2}}\big[f(2^{k}\cdot)\big]\big\|_{L^r(\R)}+ \big\|(I-\partial^2)^{\frac{\gamma}{2}}f\big\|_{L^r(\R)}\right)\\
&=C\left(2^{k(\gamma-\frac{1}{r})} \big\|(-\partial^2)^{\frac{\gamma}{2}}f\big\|_{L^r(\R)}+ \big\|(I-\partial^2)^{\frac{\gamma}{2}}f\big\|_{L^r(\R)}\right)\\
&\leq C\big(2^{k(\gamma-\frac{1}{r})}+1\big) \big\|(I-\partial^2)^{\frac{\gamma}{2}}f\big\|_{L^r(\R)},
\end{align*}
namely, \eqref{E:laplacian}.
\end{proof}

\begin{proof}[Proof of Theorem~\ref{T:comparison}]
Set $F(\xi)=\sum_{a=-n}^n \wh{\Phi}(2^a \xi)$, $\xi \in \R^n$. Then $F(\xi)=1$ for any $\xi$ satisfying $\frac{1}{2^n}\leq |\xi| \leq 2^n$. Therefore, if $j_1,\dots,j_n$ are integers and $j:=\max\{j_1,\dots,j_n\}$, then 
$F(2^{j_1-j}\xi_1,\dots,2^{j_n-j}\xi_n)=1$ on $\{(\xi_1,\dots,\xi_n): \frac{1}{2}\leq |\xi_1| \leq 2,\dots, \frac{1}{2} \leq |\xi_n| \leq 2\}$. Consequently,
$$
\prod_{\ell=1}^n \wh{\psi}(\xi_\ell) = F(2^{j_1-j}\xi_1,\dots,2^{j_n-j}\xi_n) \prod_{\ell=1}^n \wh{\psi}(\xi_\ell).
$$
Using this, we can write
\begin{align*}
&\left\|(I-\partial_1^2)^{\frac{\gamma_1}{2}} \cdots (I-\partial_n^2)^{\frac{\gamma_n}{2}}\bigg[\sigma(2^{j_1}\xi_1,\dots,2^{j_n}\xi_n) \prod_{\ell=1}^n \wh{\psi}(\xi_\ell)\bigg]\right\|_{L^r}\\
&=\left\|(I-\partial_1^2)^{\frac{\gamma_1}{2}} \cdots (I-\partial_n^2)^{\frac{\gamma_n}{2}}\bigg[\sigma(2^{j_1}\xi_1,\dots,2^{j_n}\xi_n) F(2^{j_1-j}\xi_1,\dots,2^{j_n-j}\xi_n) \prod_{\ell=1}^n \wh{\psi}(\xi_\ell)\bigg]\right\|_{L^r}\\
&\leq C\sum_{a=-n}^n\sum_{\{i_1,\dots,i_k\} \subseteq \{1,\dots,n\}} \\
&\left\|(-\partial_{i_1}^2)^{\frac{\gamma_{i_1}}{2}} \cdots (-\partial_{i_k}^2)^{\frac{\gamma_{i_k}}{2}}\bigg[\sigma(2^{j_1}\xi_1,\dots,2^{j_n}\xi_n) \wh{\Phi}(2^{j_1-j+a}\xi_1,\dots,2^{j_n-j+a}\xi_n) \prod_{\ell=1}^n \wh{\psi}(\xi_\ell)\bigg]\right\|_{L^r}.\\
\end{align*}
Using the estimate~\eqref{E:laplacian} in variables $i_1,\dots,i_k$ and inequality \eqref{E:identity} in the remaining variables, we estimate the corresponding term in the last expression by a constant multiple of
\begin{align*}
& \bigg[ \prod_{s=1}^n \big(1+2^{(j_{i_s}-j+a)(\gamma_{i_s}-\frac{1}{r})}\big) \bigg] 
\left\|(I-\partial_{1}^2)^{\frac{\gamma_{1}}{2}} \cdots (I-\partial_{n}^2)^{\frac{\gamma_{n}}{2}} \bigg[\sigma(2^{j-a}\xi) \wh{\Phi}(\xi)\bigg]\right\|_{L^r}\\
&\leq C\big(1+2^{n\max_{\ell=1,\dots,n}(\gamma_\ell - \frac{1}{r})}\big)^n
\left\|(I-\Delta)^{\frac{\gamma_{1}+\dots+\gamma_n}{2}} \bigg[\sigma(2^{j-a}\xi) \wh{\Phi}(\xi)\bigg]\right\|_{L^r}\\
&\leq C \sup_{m\in \mathbb Z} \left\|(I-\Delta)^{\frac{\gamma_{1}+\dots+\gamma_n}{2}} \bigg[\sigma(2^{m}\xi) \wh{\Phi}(\xi)\bigg]\right\|_{L^r}.
\end{align*}
This implies~\eqref{E:inequality}.
\end{proof}

\section{Examples and Remarks}\label{S:examples}

Next we discuss examples that indicate  the sharpness of Theorem~\ref{ThmMain}. 
As stated, the sufficient condition presented in   Theorem~\ref{ThmMain} is optimal in the sense that if    $T_\si$ is  bounded from $L^p(\rn)$ to itself for all $\si$
satisfying \eqref{20000}, then   \eqref{Marc-con} holds. This has been observed, at least in the two-dimensional case with both smoothness parameters equal, by Carbery and Seeger~\cite[remark after Proposition 6.1]{CSTAMS}. We provide an example in the spirit of theirs, given by an explicit closed-form expression and valid in all dimensions $n \geq 2$. 

\begin{example}\label{example}
Given $\alpha \in (0,1)$, consider the function
$$
\sigma(\xi,\eta)=\varphi(|\eta|) e^{-\frac{\xi^2}{2}} |\eta|^{i\xi} (\log |\eta|)^{-\alpha}, \quad (\xi,\eta)\in \R \times \R^{n-1}=\R^n,
$$
where $\varphi$ is a smooth function on the line such that $0\leq \varphi \leq 1$, $\varphi=0$ on $(-\infty,8]$ and $\varphi=1$ on $[9,\infty)$. Then

\noindent
\textup{(i)} $\sigma$ satisfies~\eqref{20000}, with $r$ large enough, whenever $\gamma_1=\alpha$ and $\gamma_2, \dots, \gamma_n$ are arbitrary positive real numbers;

\noindent
\textup{(ii)} $\sigma$ is an $L^p$ Fourier multiplier for a  given $1<p<\infty$ if and only if $\alpha >|\frac{1}{p}-\frac{1}{2}|$.
\end{example}

The previous example indicates that condition~\eqref{20000} does not guarantee that $T_\sigma$ is $L^p$ bounded unless all indices $\gamma_1, \dots, \gamma_n$ in~\eqref{20000} are larger than $|\frac{1}{p}-\frac{1}{2}|$. In particular, for a given $i \in \{1,\dots,n\}$, one does not have boundedness on the critical line $\gamma_i=|\frac{1}{p}-\frac{1}{2}|$, no matter how large the remaining parameters are. 

Let us now verify the statement of part \textup{(i)} of Example~\ref{example}. We shall first prove that 
\begin{equation}\label{E:mixed}
\sup_{k, \ell \in \mathbb Z} \|(I-\partial_\xi^2)^{\frac{\alpha}{2}}(I-\Delta_\eta)^{\frac{s}{2}}[\wh{\psi}(\xi) \wh{\Phi}(\eta)\sigma(2^k\xi, 2^\ell \eta)]\|_{L^r}<\infty
\end{equation}
for any $s>0$ and $r> 1$. Here, $\Phi$ denotes a Schwartz function on $\R^{n-1}$ whose Fourier transform is supported in the set $\{\eta \in \R^{n-1}: \frac{1}{2}\leq |\eta|\leq 2\}$ and which satisfies $\sum_{\ell \in \mathbb Z} \wh{\Phi}(2^\ell \eta)=1$ for all $\eta \neq 0$. Indeed, for any $k, \ell \in \mathbb Z$, $\ell \geq 3$, and for any given nonnegative integer $m$, we have
$$
\|\wh{\psi}(\xi) \wh{\Phi}(\eta)\sigma(2^k\xi, 2^\ell \eta)\|_{L^r} \leq C \ell^{-\alpha}
$$
and
$$
\|(I-\partial_\xi)^{\frac{1}{2}}(I-\Delta_\eta)^{\frac{m}{2}}[\wh{\psi}(\xi) \wh{\Phi}(\eta)\sigma(2^k\xi, 2^\ell \eta)]\|_{L^r} \leq C \ell \cdot \ell^{-\alpha},
$$
where the constant $C$ is independent of $k$ and $\ell$. Interpolating between these two estimates, we obtain
$$
\|(I-\partial_\xi)^{\frac{\alpha}{2}}(I-\Delta_\eta)^{\frac{\alpha m}{2}}[\wh{\psi}(\xi) \wh{\Phi}(\eta)\sigma(2^k\xi, 2^\ell \eta)]\|_{L^r} \leq C.
$$
Notice also that the last inequality in fact holds for all integers $k$, $\ell$, since the function $\wh{\psi}(\xi) \wh{\Phi}(\eta)\sigma(2^k\xi, 2^\ell \eta)$ is identically equal to $0$ if $\ell \leq 2$. Hence, we have
$$
\sup_{k, \ell \in \mathbb Z} \|(I-\partial_\xi)^{\frac{\alpha}{2}}(I-\Delta_\eta)^{\frac{\alpha m}{2}}[\wh{\psi}(\xi) \wh{\Phi}(\eta)\sigma(2^k\xi, 2^\ell \eta)]\|_{L^r} \leq C,
$$
and interpolating between variants of this estimate corresponding to different values of $m$, we obtain~\eqref{E:mixed} for any $s>0$. Now, part \textup{(i)} of Example~\ref{example} follows by an application of Theorem~\ref{T:comparison} in the variable $\eta$.

Let us finally focus on part \textup{(ii)} of Example~\ref{example}. If $\alpha >|\frac{1}{p}-\frac{1}{2}|$, then $\sigma$ is an $L^p$ Fourier multiplier thanks to \textup{(i)} and Theorem~\ref{ThmMain}. Let us now prove that $T_\sigma$ is not $L^p$ bounded if $\alpha \leq |\frac{1}{p}-\frac{1}{2}|$. By duality, it suffices to discuss only the case when $1<p<2$. Further, by a result of Herz and Rivi\`ere~\cite{HR}, our claim will follow if we show that $T_\sigma$ is not bounded on the mixed norm space $L^p(\R; L^2(\R^{n-1}))$.

Let $f$ be the function on $\R^n$ whose Fourier transform satisfies
$$
\wh{f}(\xi,\eta)=e^{-\frac{\xi^2}{2}} \varphi(|\eta|) |\eta|^{\frac{1-n}{2}} (\log |\eta|)^{-\frac{1}{2}} (\log \log |\eta|)^{-\beta}, \quad (\xi,\eta) \in \R \times \R^{n-1}.
$$
Using the Plancherel theorem in the variable $\eta$, it is easy to check that $f\in L^p(\R; L^2(\R^{n-1}))$ whenever $\beta >\frac{1}{2}$. Our next goal is to prove that $T_\sigma f=\mathcal F^{-1}(\sigma \wh{f})$ does not belong to $L^p(\R; L^2(\R^{n-1}))$ if $\beta \in (\frac{1}{2},\frac{1}{p}]$. Using Plancherel's theorem in the variable $\eta$ again, this is equivalent to showing that $\mathcal F^{-1}_\xi(\sigma \wh{f}\,)$ is not in $L^p(\R; L^2(\R^{n-1}))$, where $\mathcal F^{-1}_\xi$ stands for the inverse Fourier transform in the $\xi$ variable.

Observe that
\begin{align*}
\mathcal F^{-1}_\xi(\sigma \wh{f}\,)(x,\eta)
&=C e^{-\frac{1}{4}(2\pi x+\log |\eta|)^2} \varphi^2(|\eta|)|\eta|^{\frac{1-n}{2}} (\log |\eta|)^{-\alpha-\frac{1}{2}} (\log \log |\eta|)^{-\beta}\\
&\geq C \chi_{\{(x,\eta): ~x<-2, ~e^{-2\pi x-1} <|\eta|<e^{-2\pi x}\}}(x,\eta) e^{2\pi x \frac{n-1}{2}} (-x)^{-\alpha-\frac{1}{2}} (\log (-x))^{-\beta}.
\end{align*}
Therefore, 
$$
\|\mathcal F^{-1}_\xi(\sigma \wh{f}\,)\|_{L^p(\R; L^2(\R^{n-1}))} 
\geq C\left( \int_{-\infty}^{-2} (-x)^{(-\alpha-\frac{1}{2})p} (\log (-x))^{-\beta p} \,dx\right)^{\frac{1}{p}} =\infty,
$$
which yields the desired conclusion.

\medskip

Acknowledgment: We would like to thank Andreas Seeger for his useful comments.

 \end{document}